\documentclass[12pt]{article}

\usepackage{amsmath,amsxtra,amssymb,color,latexsym,epsfig,amscd,amsthm,subfigure,fancybox,epsfig}
\usepackage[mathscr]{eucal}
\usepackage{graphicx}
\usepackage{epsfig} 
\usepackage{epstopdf}
\usepackage{cases}
\def\para#1{\vskip .4\baselineskip\noindent{\bf #1}}

\newtheorem{thm}{Theorem}[section]
\newtheorem {asp}{Assumption}[section]
\newtheorem{lm}{Lemma}[section]

\newtheorem{prop}[thm]{Proposition}
\theoremstyle{definition}

\theoremstyle{remark}

\numberwithin{equation}{section}

\newcommand{\eps}{\varepsilon}

\newcommand{\M}{\mathcal{M}}
\newcommand{\F}{\mathcal{F}}

\newcommand{\N}{\mathbb{N}}
\newcommand{\PP}{\mathbb{P}}

\newcommand{\R}{\mathbb{R}}

\numberwithin{equation}{section}
\newcommand{\1}{\boldsymbol{1}}

\newcommand{\bed}{\begin{displaymath}}
\newcommand{\eed}{\end{displaymath}}
\newcommand{\bea}{\bed\begin{array}{rl}}
\newcommand{\eea}{\end{array}\eed}
\newcommand{\ad}{&\!\!\!\disp}

\newcommand{\barray}{\begin{array}{ll}}
\newcommand{\earray}{\end{array}}

\def\disp{\displaystyle}
\newcommand{\rr}{{\Bbb R}}

\def\bar{\overline}
\def\hat{\widehat}
\def\a.s{\text{\;a.s.\;}}

\begin{document}
\title{Conditions for Permanence and Ergodicity of
Certain Stochastic Predator-Prey Models}
\author{N.H. Du,\thanks{Department of Mathematics, Mechanics and
Informatics, Hanoi National University,
 334 Nguyen Trai, Thanh Xuan, Hanoi Vietnam, dunh@vnu.edu.vn. This research was
supported in part by
NAFOSTED
n$_0$ 101.02 - 2011.21.}\and
N.H. Dang,\thanks{Department of Mathematics, Wayne State University, Detroit, MI
48202, USA,
dangnh.maths@gmail.com. This
research was supported in part by the National Science Foundation under grant DMS-1207667.
} \and
 G. Yin\thanks{Corresponding author: Department of Mathematics, Wayne State University, Detroit, MI
48202, USA, gyin@math.wayne.edu.
This
research was supported in part by the National Science Foundation under grant DMS-1207667.
}}
\maketitle

\begin{abstract} This work derives sufficient conditions
for the permanence and ergodicity of a stochastic predator-prey model with Beddington-DeAngelis
functional response.
The conditions obtained in fact are very close to the necessary conditions.
Both non-degenerate and degenerate diffusions are considered.
One of the distinctive features of our results
is that our results enables
characterization of the support of a unique invariant probability measure. It proves the convergence in
 total variation norm of the transition probability to the invariant measure.
 Comparisons to existing literature and related matters to other stochastic predator-prey
 models are also given.



\end{abstract}

\newpage

\setlength{\baselineskip}{0.28in}

\section{Introduction}\label{sec:int}

 This paper focuses on stochastic predator-prey models with
 Beddington-DeAngelsis functional response.
 In ecology,
 a functional response is the intake rate of a consumer as a function of food density. It is associated
with the numerical response that is the reproduction rate of a consumer as a function of food density.
 Holling \cite{Holling} initiated the study of functional response, where he introduced several types of such responses.
The so-called Holling type II functional response is characterized by a decelerating intake rate following from the assumption
that the consumer is limited by its capacity to process food. Similar to Holling-type functional
response with an extra term describing mutual interference
by predators,
Beddington \cite{Bed} and DeAngelis  et. al. \cite{DeA} introduced
the nowadays well-known Beddington-DeAngelis functional response; see also \cite{Zhang} and the references therein.
Such a model represents most of the
qualitative features of the ratio-dependent models but avoids
the ``low densities problem.''

As the building blocks of the bio- and eco-systems,
the basic premise of the predator-prey models
is that
species compete, evolve, and disperse  for the purpose of seeking resources to sustain their struggle and existence.
Denote the two population
  sizes at time $t$
 by $x(t)$ and $y(t)$, respectively.
 Then a general deterministic model called Kolmogorov's predator-prey model takes the form
 $$\left\{ \barray
 \ad \dot x(t)= xf(x,y),\\
 \ad \dot y(t)= y g(x,y).\earray \right. $$
When $f(x,y)=b- py$ and $g(x,y)=cx -d$, one gets the so-called Lotka-Volterra model.

In addition to the study of deterministic models,
 stochastic predator-prey models have received increasing and resurgent attention.
 Stochastic models can be considered
as the above systems subject to Brownian motion perturbations.
Rudnicki \cite{RR} provided a detailed analysis for stability in distribution of a stochastic
 Lotka-Volterra model. Meanwhile, Mao et. al. \cite{MSR} and Du and Sam \cite{DS} studied general
 stochastic Lotka-Volterra models using Lyapunov-type functions and exponential martingale inequalities.
 Recently, Lotka-Volterra models in random environment have also gained much attention \cite{ZhuY-lot}.
 In addition, there is a resurgent interests in treating evolutionary games \cite{HS98}, in which Lotka-Volterra type equations
 are one of the central models.
Concerning different functional responses,
references \cite{LSJJ} and \cite{LW} dealt with the stochastic predator-prey
model with Holling functional response of the form
\begin{equation}\label{e1.0}
\begin{cases}
dx(t)=x(t)\big(a_1-b_1x(t)-\dfrac{c_1y(t)}{1+x(t)}\big)dt+\alpha x(t)dB_1(t), \\
dy(t)=y(t)\big(-a_2-b_2y(t)+\dfrac{c_2x(t)}{1+x(t)}\big)dt+\beta y(t)dB_2(t),
\end{cases}
\end{equation}
where $a_i$, $b_i$, $c_i$, $\alpha$, and $\beta$ are appropriate constants, and $B_i(\cdot)$ are
standard Brownian motions.
Ji et. al. \cite{JJS} studied the predator-prey model with modified Leslie-Gower and Holling type II schemes
with stochastic perturbation; see also  \cite{JJL} in which  stochastic ratio-dependent predator-prey models were considered.
Moreover, several stochastic models with the well-known Beddington-DeAngelsis functional response were also studied in  \cite{JJ, LW2,TDK}.
In ecology models, an important concept is
stochastic permanence, which indicates that the species will survive
forever. Much effort has been devoted to
finding conditions needed for stochastic permanence.
In some of the aforementioned papers,
 using suitable Lyapunov-type functions,
 some conditions for extinction or permanence
 were also provided and ergodicity was  investigated; see \cite{JJ, LW}.
However, as shown later in
Section \ref{sec:dis}
of this paper, their conditions are restrictive and not close to a necessary condition.
In other words, there is a considerably large set of parameters  satisfying neither their conditions for extinction nor for permanence.
Moreover, their results are not applicable to degenerate cases.
Thus, although interesting, their work left a sizable gap.
One of the main goals of this paper is to
close this gap. We aim to providing
a sufficient and almost necessary condition for permanence (as well as ergodicity) for the following
model with Beddington-DeAnglesis functional response,
\begin{equation}\label{me}
\begin{cases}
dx(t)=x(t)\big(a_1-b_1x(t)-\dfrac{c_1y(t)}{m_1+m_2x(t)+m_3y(t)}\big)dt+\alpha x(t)dB_1(t), \\
dy(t)=y(t)\big(-a_2-b_2y(t)+\dfrac{c_2x(t)}{m_1+m_2x(t)+m_3y(t)}\big)dt+\beta y(t)dB_2(t),
\end{cases}
\end{equation}
where $a_i, b_i, c_i, m_i$ are positive constants for $i=1,2$, $m_3\geq0$, $\alpha\ne0,\beta\ne0$,
and $B_1(\cdot), B_2(\cdot)$ are two mutually independent Brownian motions.
When $m_3=0,$ the functional response is said to be of Holling type II.
Moreover, in this paper, we also consider the degenerate case $B_1(\cdot)=B_2(\cdot)$.

The rest of the paper is arranged as follows.
Section \ref{sec:thr} derives a threshold that is used to determine extinction and permanence.
To establish the desired result, after considering the dynamics on the boundary, we obtain
 a threshold $\lambda$ that enables us to determine the asymptotic behavior of the solution.
In particular, it is shown that if $\lambda<0$, the predator will eventually die out.
In case $\lambda>0$, the solution  converges to a stationary distribution in total variation norm. Moreover, ergodicity is established.
Section \ref{sec:thr} concentrates on non-degenerate case,
whereas Section \ref{sec:deg} treats the degenerate case $B_1(\cdot)=B_2(\cdot)$.
In the degenerate case, under usual conditions imposed on the Lie algebra generated by the drift and the diffusion coefficients,
we investigate the controllability of the associated control systems and used certain results in
\cite{WK} to prove analogous results to the nondegenerate case, namely,
 the existence and uniqueness of an invariant probability measure as well as the convergence in total variation of the transition probability.
 Moreover, the support of the invariant measure is described.
Finally, Section \ref{sec:dis} provides further discussion and insight. Among other things, it
points out that the techniques used in this paper
 can be  applied to other stochastic predator-prey models.

\section{Threshold Between Extinction and Permanence}\label{sec:thr}
Let $(\Omega,\F,\{\F_t\}_{t\geq0},\PP)$ be a complete filtered probability space with the filtration $\{\F_t\}_{t\geq 0}$ satisfying the usual condition,  i.e., it is increasing and right continuous while $\F_0$ contains all $\PP-$null sets. Let $B_1(t)$ and $B_2(t)$ be two $\F_t$-adapted, mutually independent Brownian motions.
It is well known that for any initial value $(x(0),y(0))\in\R^{2,\circ}_+$ (the interior of $\rr^2_+$),
there exists a unique global solution to \eqref{me} that remains in $\R^{2,\circ}_+$ almost surely (see \cite{JJ}). To proceed,
we first consider  the equation on the boundary,
\begin{equation}\label{e2.1}
d\varphi(t)=\varphi(t)(a_1-b_1\varphi(t))dt+\alpha \varphi(t)dB_1(t).
\end{equation}
By comparison theorem, it is easy to check that $x(t)\leq \varphi(t)$ $\forall t\geq0$ a.s. provided that $x(0)= \varphi(0)>0$ and $y(0)>0$.
If $a_1\leq\alpha^2/2$, we can easily verify item (2) of \cite[Theorem 3.1, p. 447]{IW} to show that
$\lim\limits_{t\to\infty}\varphi(t)=0$ a.s.  Hence, $\lim\limits_{t\to\infty}x(t)=0$ a.s., which
implies $\lim\limits_{t\to\infty}y(t)=0$ almost surely (a.s.).
For this reason, in the sequel, we
suppose that $a_1>\alpha^2/2$ throughout the rest of the paper.

Defining  $\theta(t)=\ln \varphi(t)$, equation \eqref{e2.1} becomes
\begin{equation}\label{e2.2}
d\theta(t)=\Big(a_1-\dfrac{\alpha^2}2-b_1\exp\big(\theta(t)\big)\Big)dt+\alpha dB_1(t).
\end{equation}
By solving the Fokker-Planck equation, it is shown that the process $\theta(t)$ has a unique stationary distribution with density given by
$f^*(x)=C\exp\Big(qx-a\exp(x)\Big),$
 where $q=\dfrac{2a_1}{\alpha^2}-1>0$, $a=\dfrac{2b_1}{\alpha^2}>0$, and $C$ is the normalizing constant.
Since $\theta(t)=\ln \varphi(t)$, it can be   easily  seen that $\varphi(t)$ has a unique stationary distribution $\mu_-(\cdot)$ with density $\phi^*(x)=Cx^{q-1}e^{-ax}, x>0$. It turns out that
 that $C=a^q/\Gamma(q)$ with $\Gamma(\cdot)$ being the Gamma function and that $\mu_-(\cdot)$ is the Gamma distribution with parameters $q$ and $a$.

\noindent By the strong law of large number type result \cite[Theorem 3.16, p. 46]{AS}, we deduce that
\begin{equation}\label{e2.3}
\lim\limits_{t\to\infty}\dfrac1t \int_0^t\varphi^p(s)ds=\dfrac{a^q}{\Gamma(q)}\int_{0}^\infty x^{p+q-1}e^{-ax}dx=\dfrac{\Gamma(p+q)}{a^p\Gamma(q)}:=K_p<\infty\text{ a.s }\,\forall p>0.
\end{equation}
In particular,
with $p=1$, $K_1=\dfrac{q}a=	\dfrac{a_1-\alpha^2/2}{b_1}$.
This property implies that
$$\lim\limits_{t\to\infty}\dfrac1t\ln \varphi(t)=\lim\limits_{t\to\infty}\bigg(\dfrac1t\int_0^t\Big(a_1-\dfrac{\alpha^2}2-b_1\varphi(s)\Big)ds\bigg)+\alpha\lim\limits_{t\to\infty}\dfrac{B_1(t)}t=0.$$
Consequently,
\begin{equation}\label{e2.4}
\limsup\limits_{t\to\infty}\dfrac1t\ln x(t)\leq0,
\end{equation}
and
\begin{equation}\label{e2.5}
\limsup\limits_{t\to\infty}\dfrac1t \int_0^tx^p(s)ds\leq K_p.
\end{equation}
Let $\psi(t)$ be the solution to
$$
d\psi(t)=\psi(t)\big(-a_1+\dfrac{c_2}{m_2}-b_2\psi(t)\big)dt+\beta \psi(t)dB_2(t).
$$
Then $y(t)\leq \psi(t)\,\forall t\geq0$ a.s. provided $y(0)=\psi(0)>0$.
Hence, with probability 1
\begin{equation}\label{e2.7}\footnote{In J. Bao and J. Shao, Permanence and Extinction of Regime-Switching Predator-Prey Models, {\it SIAM J. Math. Anal.}, 48(1), 725–739,
they said that an additional condition needs to be added to obtain \eqref{e2.7}, namely
$-a_2-\dfrac{\beta^2}2+\dfrac{c_2}{m_2}>0$ because it is the condition for $\psi(t)$ to be an ergodic process.

However, \eqref{e2.7} can be obtained without this condition.
Indeed, suppose that $-a_2-\dfrac{\beta^2}2+\dfrac{c_2}{m_2}\leq0$.
Let $\tilde\psi(t)$ be the solution to
$$
d\tilde\psi(t)=\tilde\psi(t)\big(-a_1+\dfrac{c_2}{m_2}\big)dt+\beta\tilde \psi(t)dB_2(t).
$$
  We have that
$$\lim\limits_{t\to\infty}\dfrac{\ln\tilde\psi(t)}t=\lim\limits_{t\to\infty}\left(\dfrac{\ln\tilde\psi(0)}t-a_2-\dfrac{\beta^2}2+\dfrac{c_2}{m_2}+\dfrac{B_2(t)}t\right)=-a_2-\dfrac{\beta^2}2+\dfrac{c_2}{m_2}\leq0\text{ a.s.}$$
By the comparison theorem,
$$\limsup\limits_{t\to\infty}\dfrac1t\ln y(t)\leq \limsup\limits_{t\to\infty}\dfrac1t\ln \psi(t)\leq \lim\limits_{t\to\infty}\dfrac1t\ln \tilde \psi(t)\leq 0\text{ a.s.}
$$
Thus, \eqref{e2.7} holds without an additional condition.
}
\limsup\limits_{t\to\infty}\dfrac1t\ln y(t)\leq 0 ,
\end{equation}

and
\begin{equation}\label{e2.8}
\limsup\limits_{t\to\infty}\dfrac1t\int_0^ty^p(s)ds\leq \hat K_p\,\mbox{ for some constant } \hat K_p>0.
\end{equation}
Define the threshold
$$\lambda:=-a_2-\dfrac{\beta^2}2+\int_0^\infty\dfrac{c_2x}{m_1+m_2x}\mu_-(dx)=-a_2-\dfrac{\beta^2}2+\dfrac{a^q}{\Gamma(q)}\int_{0}^{\infty}\dfrac{c_2x^{q}e^{-ax}}{m_1+m_2x}dx.$$

\begin{thm}\label{thm2.1}
If $\lambda<0$, then the predator is eventually extinct, that is, $\lim\limits_{t\to\infty}y(t)=0$ a.s. Moreover, as $t\to\infty$ the distribution of $x(t)$ converges weakly  to $\mu_-(\cdot)$ that is the Gamma distribution with parameters $q=\dfrac{2a_1}{\alpha^2}-1$ and $a=\dfrac{2b_1}{\alpha^2}$, respectively.
 \end{thm}

\begin{proof}
Let $\bar y(t)$ be the solution to the equation
\begin{equation}
d \bar y(t)=\bar y(t)\big(-a_2-b_2\bar y(t)+\dfrac{c_2\varphi(t)}{m_1+m_2\varphi(t)}\big)dt+\beta \bar y(t)dB_2(t),
\end{equation}
where $\varphi(t)$ is the solution to \eqref{e2.1}.
By comparison theorem, $y(t)\leq\bar y(t)$ a.s. given that $\varphi(0)=x(0), \bar y(0)= y(0)$.
In view of the It\^o formula and the ergodicity of $\varphi(t)$,
\begin{equation}
\begin{aligned}
\limsup\limits_{t\to\infty}\dfrac1t\ln\bar y(t)=&\limsup\limits_{t\to\infty}\bigg(\dfrac1t\int_0^t\big(-a_2-\dfrac{\beta^2}2-b_2\bar y(s)+\dfrac{c_2\varphi(s)}{m_1+m_2\varphi(s)}\big)ds+\beta\dfrac{B_2(t)}t\bigg)\\
\leq&\lim\limits_{t\to\infty}\dfrac1t\int_0^t\big(-a_2-\dfrac{\beta^2}2
+\dfrac{c_2\varphi(s)}{m_1+m_2\varphi(s)}\big)ds+\beta\lim\limits_{t\to\infty}\dfrac{B_2(t)}t
= \lambda<0\ \hbox{ a.s.}
\end{aligned}
\end{equation}
That is, $y(t)$ converges to 0 at an exponential rate almost surely.
The remaining
part of the assertion can be proved by the arguments in \cite[Lemma 7]{RR}.
\end{proof}

\begin{thm}\label{thm2.2}
If $\lambda>0$, the process $(x(t), y(t))$ has an invariant probability measure concentrated on $\R^{2,\circ}_+$.
\end{thm}

\begin{proof} For any initial value $(x(0), y(0))\in\R^{2,\circ}_+$, we have
\begin{equation}\label{e2.11}
\begin{aligned}
\dfrac1t\ln y(t)=&-\dfrac1t\int_0^tb_2y(s)ds+\dfrac1t\int_0^t\Big(-a_2-\dfrac{\beta^2}2+\dfrac{c_2\varphi(s)}{m_1+m_2\varphi(s)}\Big)ds\\
&-\dfrac1t\int_0^t\Big(\dfrac{c_2\varphi(s)}{m_1+m_2\varphi(s)}-\dfrac{c_2x(s)}{m_1+m_2x(s)}\Big)ds\\
&-\dfrac1t\int_0^t\Big(\dfrac{c_2x(s)}{m_1+m_2x(s)}-\dfrac{c_2x(s)}{m_1+m_2x(s)+m_3y(s)}\Big)ds+\beta \dfrac{B_2(t)}t\\
\geq& \dfrac1t\int_0^t\Big(-a_2-\dfrac{\beta^2}2+\dfrac{c_2\varphi(s)}{m_1+m_2\varphi(s)}\Big)ds\\
&-\dfrac1t\int_0^t\Big(\dfrac{c_2}{m_1}(\varphi(s)-x(s))+\big(\dfrac{c_2m_3}{m_1m_2}+b_2\big)y(s)\Big)ds+\beta \dfrac{B_2(t)}t.
\end{aligned}
\end{equation}
Letting $t\to\infty$,  \eqref{e2.7} and \eqref{e2.11} yield that
\begin{equation}\label{e2.12}
\liminf\limits_{t\to\infty}\dfrac1t\int_0^t\Big(\dfrac{c_2}{m_1}
(\varphi(s)-x(s))+\big(\dfrac{c_2m_3}{m_1m_2}+b_2\big)y(s)\Big)ds\geq\lambda\ \hbox{ a.s.}
\end{equation}
Similarly, we have
\begin{equation}\label{e2.13}
\begin{aligned}
\dfrac1t\ln x(t)=&\dfrac1t\int_0^t\Big(a_1-\dfrac{\alpha^2}2-b_1\varphi(s)\Big)ds\\
&+\dfrac1t\int_0^t\Big(b_1(\varphi(s)-x(s))-\dfrac{c_1y(s)}{m_1+m_2x(s)+m_3y(s)}\Big)ds+\alpha\dfrac{B_1(t)}t\\
\geq& \dfrac1t\int_0^t\Big(a_1-\dfrac{\alpha^2}2-b_1\varphi(s)\Big)ds+\dfrac1t\int_0^t\Big(b_1(\varphi(s)-x(s))-\dfrac{c_1y(s)}{m_1}\Big)ds+\alpha\dfrac{B_1(t)}t.
\end{aligned}
\end{equation}
It follows from \eqref{e2.3}, \eqref{e2.4}, and \eqref{e2.13} that
\begin{equation}\label{e2.14}
\liminf\limits_{t\to\infty}\dfrac1t\int_0^t\Big(-b_1(\varphi(s)-x(s))+\dfrac{c_1}{m_1}y(s)\Big)ds\geq0\ \hbox{ a.s.}
\end{equation}
Dividing both sides of \eqref{e2.12} and \eqref{e2.14} by $\dfrac{c_2}{m_1}$ and $b_1$, respectively,
and adding them side by side, we have
\begin{equation}
\liminf\limits_{t\to\infty}\dfrac1t\int_0^t y(s)ds\geq \dfrac{b_1m^2_1m_2\lambda}{c_1c_2m_2+b_1c_2m_1m_3+b_1b_2m_1^2m_2}=:\bar m>0\ \hbox{ a.s.}
\end{equation}
For $0<\hbar<\bar m<H<\infty$,  H\"older's inequality yields that
$$\dfrac1t\int_0^t\1_{\{y(s)\geq\hbar\}}y(s)ds\leq \bigg(\dfrac1t\int_0^t\1_{\{y(s)\geq\hbar\}}ds\bigg)^{\frac12}\bigg(\dfrac1t\int_0^ty^2(s)
ds\bigg)^{\frac12},$$
which implies that
\begin{equation}\label{e2.16}
\begin{aligned}
\liminf\limits_{t\to\infty}&\dfrac1t\int_0^t\1_{\{y(s)\geq\hbar\}}ds\geq \Big(\liminf\limits_{t\to\infty}\dfrac1t\int_0^t\1_{\{y(s)\geq\hbar\}}y(s)ds\Big)^2 \Big(\limsup\limits_{t\to\infty}\dfrac1t\int_0^ty^2(s)ds\Big)^{-1}\\
\geq& \Big(\liminf\limits_{t\to\infty}\dfrac1t\int_0^ty(s)ds-\hbar\Big)^2 \Big(\limsup\limits_{t\to\infty}\dfrac1t\int_0^ty^2(s)ds\Big)^{-1}
\geq \dfrac{(\bar m-\hbar)^2}{\hat K_2}\ \hbox{ a.s.}
\end{aligned}
\end{equation}
In addition, \eqref{e2.5} and \eqref{e2.8} imply that
\begin{equation}\label{e2.17}
\begin{aligned}
\limsup\limits_{t\to\infty}\dfrac1t\int_0^t\1_{\{y(s)\geq H\}}ds\leq&\dfrac1H\limsup\limits_{t\to\infty}\dfrac1t\int_0^ty(s)ds\leq \dfrac{\hat K_1}{H},\ \hbox{ a.s.,}\\
\limsup\limits_{t\to\infty}\dfrac1t\int_0^t\1_{\{x(s)\geq H\}}ds\leq&\dfrac1H\limsup\limits_{t\to\infty}\dfrac1t\int_0^tx(s)ds\leq \dfrac{K_1}{H} \ \hbox{ a.s.}
\end{aligned}
\end{equation}
It follows from \eqref{e2.16} and \eqref{e2.17} that for $\hbar<\dfrac{\bar m}2, H> \dfrac{8(K_1+\hat K_1)\hat K_2}{\bar m^2}$,
\begin{equation}\label{e2.18}
\liminf\limits_{t\to\infty}\dfrac1t\int_0^t\1_{\{(x(s), y(s))\in A\}}ds\geq \dfrac{(\bar m-\hbar)^2}{\hat K_2}-\dfrac{K_1+\hat K_1}H>\dfrac{\bar m^2}{8\hat K_2}\  \hbox{ a.s.,}
\end{equation}
where $A=\{(x, y): 0< x\leq H, \hbar\leq y\leq H\}.$
By virtue of Fatou's Lemma, we have
\begin{equation}\label{e2.19}
\liminf\limits_{t\to\infty}\dfrac1t\int_0^tP(s, (x, y), A)ds\geq \dfrac{\bar m^2}{8\hat K_2}\,\forall (x,y)\in \R_+^{2,\circ},
\end{equation}
where $P(t, (x, y), \cdot)$ is the transition probability of $(x(t), y(t))$.
By the invariance of $\M=\{x\geq0, y>0\}$ under equation \eqref{me}, we can consider the Markov process $(x(t), y(t))$ on the state space $\M$.
It is easy to show that $(x(t), y(t))$ has the Feller property. Thus,
inequality \eqref{e2.19} implies that there is an invariant probability measure $\mu^*$ on $\M$; see
\cite{MT}.
Since $y(t)\to0$ provided that $ x(0)=0$,  $\lim_{t\to\infty}P(t, (0, y), K)=0$  for all compact set $K\subset \M$.
Thus, we must have $\mu^*(\{x=0, y>0\})=0$ (equivalently $\mu^*(\R_+^{2,\circ})=1$). Furthermore,
by  the invariance of  $\R_+^{2,\circ}$,  $\mu^*$ is an invariant probability measure of $(x(t), y(t))$ on $\R_+^{2,\circ}$.
\end{proof}

Since $B_1(\cdot)$ and $B_2(\cdot)$ are independent, the diffusion is non-degenerate. It is well known that the existence of an invariant probability measure is equivalent to positive recurrence. Hence, the invariant probability is unique and the strong law of large numbers holds; see \cite[Theorems 3.1, 3.3]{RK}.
We have the following result.

\begin{thm}\label{thm2.3}
If $\lambda>0$, \eqref{me} has a unique invariant probability measure $\mu^*$ with support $\R^{2,\circ}_+$.
Moreover,
\begin{itemize}
\item[{\rm (a)}] For any $\mu^*$-integrable $f(x, y): \R^{2,\circ}_+\to\R$, we have
$$\lim\limits_{t\to\infty}\dfrac1t\int_0^tf(x(s), y(s))ds=\int f(x,y)\mu^*(dx, dy) a.s.\,\forall (x(0), y(0))\in\R^{2,\circ}_+.$$
\item[{\rm (b)}] $\lim\limits_{t\to\infty}\|P(t, (x, y), \cdot)-\mu^*(\cdot)\|=0\,\forall (x,y)\in\R^{2,\circ}_+$ where $\|\cdot\|$ is the total variation norm.
\end{itemize}
\end{thm}

\begin{proof}
Assertion (a) was proved in \cite[Theorem 3.3]{RK};
we
refer to \cite[Proposition 5.1]{IK} or \cite{LB} for the proof of assertion (b).
\end{proof}

As a direct corollary of Theorem \ref{thm2.3}, if $\lambda>0$, system \eqref{me} is stochastically permanent in the sense that for any $\eps>0$, there is some $\delta\in(0,1)$ such that
$\liminf\limits_{t\to\infty}P(t, x, y, [\delta, \delta^{-1}]^2)>1-\eps$.
Moreover, it follows from \eqref{e2.5} and \eqref{e2.8} that we have the following limits.
\begin{align*}&\lim\limits_{t\to\infty}\dfrac1t\int_0^tx^p(s)ds=\int x^p\mu^*(dx, dy)\text{ a.s. }\;\forall (x(0), y(0))\in\R^{2,\circ}_+, \ p>0,\\
&\lim\limits_{t\to\infty}\dfrac1t\int_0^ty^p(s)ds=\int y^p\mu^*(dx, dy)\text{ a.s. }\; \forall (x(0), y(0))\in\R^{2,\circ}_+, \ p>0.
\end{align*}

\section{Degenerate Case}\label{sec:deg}
Suppose that $B_1(\cdot)=B_2(\cdot)=W(\cdot)$. We consider the system of equations
\begin{equation}\label{e3.0}
\begin{cases}
dx(t)=x(t)\big(a_1-b_1x(t)-\dfrac{c_1y(t)}{m_1+m_2x(t)+m_3y(t)}\big)dt+\alpha x(t)dW(t), \\
dy(t)=y(t)\big(-a_2-b_2y(t)+\dfrac{c_2x(t)}{m_1+m_2x(t)+m_3y(t)}\big)dt+\beta y(t)dW(t).
\end{cases}
\end{equation}
Owing to the symmetry of the Brownian motion, we can suppose $\alpha>0$.
Since estimates in the previous section still hold for this case, we have $\lim\limits_{t\to\infty}y(t)=0$ when  $\lambda<0$ while $x(t)$ converges weakly to the stationary distribution $\mu_-$ of $\varphi(t)$. In what follows,
we suppose $\lambda>0$ for which the process has an invariant probability measure $\mu^*$ on $\R_+^{2,\circ}$.
Putting $\xi(t)=\ln x(t)$ and $\eta(t)=\ln y(t)$,
equation \eqref{e3.0} becomes
\begin{equation}\label{e3.1}
\begin{cases}
d\xi(t)=\big(a_1-\dfrac{\alpha^2}2-b_1e^{\xi(t)}-\dfrac{c_1e^{\eta(t)}}{m_1+m_2e^{\xi(t)}+m_3e^{\eta(t)}}\big)dt+\alpha dW(t), \\
d\eta(t)=\big(-a_2-\dfrac{\beta^2}2-b_2e^{\eta(t)}+\dfrac{c_2e^{\xi(t)}}{m_1+m_2e^{\xi(t)}+m_3e^{\eta(t)}}\big)dt+\beta dW(t).
\end{cases}
\end{equation}
Denote by  $(\xi^{u,v}(t),\eta^{u,v}(t))$ the solution with initial value $(u,v)$ to \eqref{e3.1} and let $\hat P(t, (u,v), \cdot)$ be its transition probability.
Put
$$A(u, v)=\left(\begin{array}{l}a_1-\dfrac{\alpha^2}2-b_1e^{u}-\dfrac{c_1e^{v}}{m_1+
m_2e^{u}+m_3e^{v}}\\
-a_2-\dfrac{\beta^2}2-b_2e^{v}+\dfrac{c_2e^{u}}{m_1+m_2e^{u}+m_3e^{v}}
\end{array}\right)\,\mbox{ and }\,
B(u, v)=\left(\begin{array}{l}\alpha\\
\beta
\end{array}\right).$$
To proceed, we first
 recall the notion of Lie bracket.
If $X(x)=(X_1, X_2)^\top$ and $Y(x)=(Y_1, Y_2)^\top$ are vector fields on $\R^2$ then the Lie bracket $[X,Y]$ is a vector field given by
$$[X,Y]_i(x)=\Big(X_1 \frac{\partial Y_i}{\partial x_1}(x)-Y_1 \frac{\partial X_i}{\partial x_1}(x)\Big)+\Big(X_2 \frac{\partial Y_i}{\partial x_2}(x)-Y_2 \frac{\partial X_i}{\partial x_2}(x)\Big), \ i=1,2.$$
We impose the following condition.

\begin{asp}\label{asp3.1}
The Lie algebra $\L(u, v)$ generated by $A(u, v), B(u, v)$ satisfies dim$\L(u, v)=2$ at every $(u, v)\in\R^2$. In other words, the set of vectors $A, B, [A, B], [A, [A, B]], [B, [A, B]],\dots$ spans $\R^2$.
\end{asp}

This assumption appears to be satisfied for most
practical situations.
It seems to be satisfied for any
 $a_i,b_i,c_i, m_1,m_2,m_3, \alpha>0, i=1, 2$, $\beta\ne 0$ and $a_1-{\alpha^2}/2>0$,
although
verifying this assumption for our model in general involves
cumbersome calculations.
For specific parameters, the assumption can be verified by direct calculations.
%
Note that the set of $(u, v)$ at which vectors $A, B, [A, B], [A, [A, B]], [B, [A, B]],\dots$ do not span $\R^2$ is roots of a system of equations $\det (A, B)=0, \det(A, [A, B])=0, ...$ each of which is a polynomial equation of unknowns $e^u, e^v$.
Thus, we can show that there is no $(u, v)$ satisfying the above system of equations after taking into account a sufficient number of these equations.

To describe the support of the invariant measure $\mu^*$ and to prove the ergodicity of \eqref{e3.1}, we need to investigate the following control system
\begin{equation}\label{e3.2}
\left\{\begin{array}{l}\dot u_\phi(t)=\alpha\phi(t)+a_1-\dfrac{\alpha^2}2-b_1e^{u_\phi(t)}-\dfrac{c_1e^{v_\phi(t)}}{m_1+
m_2e^{u_\phi(t)}+m_3e^{v_\phi(t)}},\\
\dot v_\phi(t)=\beta\phi(t)-a_2-\dfrac{\beta^2}2-b_2e^{v_\phi(t)}+\dfrac{c_2e^{u_\phi(t)}}{m_1+m_2e^{u_\phi(t)}+m_3e^{v_\phi(t)}},
\end{array}\right.
\end{equation}
where $\phi$ is taken from the set of piecewise continuous real valued functions defined on $\R_+$.
Let $(u_\phi(t, u,v),$ $ v_\phi(t, u, v))$ be the solution to Equation \eqref{e3.2} with control $\phi$ and initial value $(u,v)$.
Denote by ${\cal O}_1^+(u, v)$ the reachable set from $(u, v)$, that is the set of $(u', v')\in\R^2$ such that there exists a $t\geq0$ and a control $\phi(\cdot)$ satisfying
$u_\phi(t, u, v)=u', v_\phi(t, u, v)=v'$.
It should be noted that Assumption \ref{asp3.1} guarantees the accessibility of \eqref{e3.2}, i.e., ${\cal O}_1^+(u, v)$ has non-empty interior for every $(u, v)\in\R^2$ (see \cite{Ju}).
We first recall some concepts introduced in  \cite{WK}.
Let $U$ be a subset of $\R^2$ satisfying the property that for any $w_1, w_2\in U$, we have $w_2\in \bar{{\cal O}^+_1(w_1)}$.
Then there is a unique maximal set $V\supset U$ such that this property still holds for $V$. Such $V$ is called a control set.
A control set $C$ is said to be invariant if $\bar{{\cal O}^+_1(w)}\subset\bar C$ for all $w\in C$.

Putting $z_\phi=v_\phi-\frac{\beta}\alpha u_\phi$, we have an equivalent system
\begin{equation}\label{e3.3}
\left\{\begin{array}{l}\dot u_\phi(t)=\alpha\phi(t)+g(u_\phi(t),z_\phi(t)),\\
\dot z_\phi(t)=h(u_\phi(t), z_\phi(t)),
\end{array}\right.
\end{equation}
where
$$g(u,z )=a_1-\dfrac{\alpha^2}2-b_1e^{u }-\dfrac{c_1e^{z }e^{\frac\beta\alpha u }}{m_1+m_2e^{u }+m_3e^{z }e^{\frac\beta\alpha u }},$$
and
$$h(u, z)=-\Big(a_2+\dfrac{\beta^2}2+\dfrac\beta\alpha(a_1-\dfrac{\alpha^2}2)\Big)-b_2e^{z}e^{\frac\beta\alpha u}+\dfrac\beta\alpha b_1e^{u}+\dfrac{c_2e^{u}+\dfrac\beta\alpha c_1e^{z+\frac\beta\alpha u}}{m_1+m_2e^{u}+m_3e^{z+\frac\beta\alpha u}}.$$
Denote by ${\cal O}^+_2(u, z)$ the set of $(u', z')\in\R^2$ such that there is a $t>0$ and a control $\phi(\cdot)$ such that
$u_\phi(t, u, z)=u', z_\phi(t, u, v)=z'$.

\para{Claim 1.} For any $u_0, u_1, z_0\in\R$ and $\eps>0$, there exists a control $\phi$ and some $T>0$ such that
$u_\phi(T, u_0, z_0)=u_1$, $|z_\phi(T, u_0, z_0)-z_0|<\eps$.

For the proof,
suppose that $u_0<u_1$ and let $\rho_1=\sup\{|g(u, z)|, |h(u, z)|: u_0\leq u\leq u_1, |z-z_0|\leq\eps\}.$ We choose  $\phi(t)\equiv\rho_2$ with $\left(\alpha\rho_2\rho_1^{-1}-1\right)\eps\geq u_1-u_0$. It is easy to check that with this control, there is a $T\in[0,\eps\rho_1^{-1}]$ such that $u_\phi(T, u_0, z_0)=u_1$, $|z_\phi(T, u_0, z_0)-z_0|<\eps$.
If $u_0>u_1$, we can construct $\phi(t)$ similarly.

\para{Claim 2.} For any $z_0>z_1$, there is a $u_0\in\R$, a control $\phi$, and some $T>0$ such that $z_\phi(T, u_0, z_0)=z_1$ and that $u_\phi(t, u_0, z_0)=u_0\,\forall\, 0\leq t\leq T$.

Indeed, if $\beta>0$ and $-u_0$ is sufficiently large, there is a $\rho_3>0$ such that
$h(u_0, z)<-\rho_3\,\forall z_1\leq z\leq z_0$. This property, combining with \eqref{e3.3}, implies the existence of a control $\phi$ and a $T>0$ satisfying the desired claim. In case $\beta<0$, choosing $u_0$
to be sufficiently
large,
we have the same result.

\para{Claim 3.} If $0<\beta<\alpha$, for any $z_0<z_1$, if $u_0$ is sufficiently large, $\inf_{z\in[z_0,z_1]}h(u_0, z)>0$, which implies that there is a control $\phi$ and a $T>0$ satisfying $z_\phi(T, u_0, z_0)=z_1$ and  $u_\phi(t, u_0, z_0)=u_0\,\forall 0\leq t\leq T$.

\begin{lm}\label{lm3.1}
 Suppose $\beta<0$ or $\beta\geq\alpha$. Let $ c^*:=\sup\Big\{\bar z: \sup\limits_{u\in R}\{h(u, z)\}>0\,\forall\, z\leq \bar z.\Big\}$.
Then $ c^*>-\infty$, $( c^*$ may be $\infty)$ and for any $(u, z)\in\R^2$, $\bar{{\cal O}^+_2(u, z)}\supset\{(u',z'): z'\leq c^*\}$.
\end{lm}

\begin{proof}
Note that $$\lambda=-a_2-\dfrac{\beta^2}2+\int_{0}^\infty\dfrac{c_2 x}{m_1+m_2x}\mu_-(dx)>0.$$ In view of Jensen's inequality,
$$\int_{0}^\infty\dfrac{c_2x}{m_1+m_2x}\mu_-(dx)\leq \dfrac{c_2\int_0^\infty x\mu_-(dx)}{m_1+m_2\int_0^\infty x\mu_-(dx)}=
\dfrac{c_2\big(a_1-\dfrac{\alpha^2}2\big)b_1^{-1}}{m_1+m_2\big(a_1-\dfrac{\alpha^2}2\big)b_1^{-1}}.$$
If $e^{\bar u}=\big(a_1-\dfrac{\alpha^2}2\big)b_1^{-1}$, we have
\begin{align*}
h(\bar u, z)=&\dfrac{c_1\big(a_1-\dfrac{\alpha^2}2\big)b_1^{-1}}{m_1+m_2\big(a_1-\dfrac{\alpha^2}2\big)b_1^{-1}+m_3e^z.e^{\frac{\beta}{\alpha}\bar u}}-\big(a_2+\dfrac{\beta^2}2\big)+b_2e^{z}e^{\frac\beta\alpha \bar u}+\dfrac{\frac\beta\alpha c_1e^{z+\frac\beta\alpha \bar u}}{m_1+m_2e^{\bar u}+m_3e^{z+\frac\beta\alpha\bar  u}}.
\end{align*}
Since $$\dfrac{c_1\big(a_1-\dfrac{\alpha^2}2\big)b_1^{-1}}
{m_1+m_2\big(a_1-\dfrac{\alpha^2}2\big)b_1^{-1}}-\big(a_2+\dfrac{\beta^2}2\big)>0,$$
$h(\bar u, z)>0$ when $e^{z}$ is sufficiently small.
Now we move
to the second assertion. Note that it follows directly from the continuous dependence of solutions on  initial values that if $\bar{{\cal O}^+_2(w_2)}\subset\bar{{\cal O}^+_2(w_1)}$ provided $w_2\in \bar{{\cal O}^+_2(w_1)} \;\;(\omega_1,\omega_2\in\R^2)$.
For $(u, z)\in\R^2$, define $\mathfrak{z}_{u,z}=\sup\big\{z_1: \exists u_1\mbox{ such that } (u_1, z_1)\in \bar{{\cal O}^+_2(u, z)}\big\}$.
For any $(u_1, z_1)\in\R^2$, it is easy to derive from Claims 1 and 2 that $\bar{{\cal O}^+_2(u_1, z_1)}\supset\{(u', z'): z'\leq z_1\}$.
Hence $\bar{{\cal O}^+_2(u, z)}\supset\{(u_1, z_1): z_1\leq \mathfrak{z}_{u, z}\}.$
If $\mathfrak{z}_{u, z}< c^*$, there is some $\hat u\in\R$ such that $h(\hat u, \mathfrak{z}_{u, z})>0$. Since $h(\cdot)$ is continuous, there is an $\hat z>\mathfrak{z}_{u, z}$ such that $\inf\{h(\hat u, z): z\in[\mathfrak{z}_{u, z},\hat z]\}>0$.
As a result, there is a control $\phi$ and a $T>0$ such that $z_\phi(T, \hat u, \mathfrak{z}_{u, z})=\hat z$ and $u_\phi(t, \hat u, \mathfrak{z}_{u, z})=\hat u\,\forall t\in[0, T]$.
 That is, $(\hat u,\hat z)\in{{\cal O}^+_2(\hat u,  \mathfrak{z}_{u, z})}\subset \bar{{\cal O}^+_2(u, z)}$, which contradicts  the definition of $\mathfrak{z}_{u, z}$.
The proof is complete.
\end{proof}

\begin{prop}\label{prop3.1}
The control system \eqref{e3.2} has only one invariant control set $C$. If $0<\beta<\alpha$, $C=\R^2$. If $\beta<0$ or $\beta\geq\alpha$, $C=\{(u, v): v-\dfrac\beta\alpha u\leq c^*\}.$
\end{prop}

\begin{proof}
If $0<\beta<\alpha$, it follows from Claims 1, 2, and 3 that for any $(u_1, z_1), (u_2, z_2)\in\R^2$, $(u_2, z_2)\in \bar{{\cal O}^+_2(u_1, z_1)}.$ Hence, for any $(u_1, v_1), (u_2, v_2)\in\R^2$, we have $(u_2, v_2)\in \bar{{\cal O}^+_1(u_1, v_1)}.$ This implies that $\R^2$ is an unique invariant control set.
Now, consider the case $\beta<0$ or $\beta\geq\alpha$ for which the conclusion of this proposition is a direct corollary of Lemma \ref{lm3.1} if $c^*=\infty$.
If $c^*<\infty$,  it is seen from the definition of $c^*$ that $h(u, c^*)\leq0\,\forall u\in\R$. Consequently, for all control $\phi$, we have $z_\phi(t, u, z)\leq c^*\,\forall t\geq0$ provided that $z\leq c^*$. In other words, $ \bar{{\cal O}^+_2(u, z)}\subset \{(u', z'): z'\leq c^*\}$.
This claim combined with Lemma \ref{lm3.1} implies that
$\bar{{\cal O}^+_2(u, z)}=\{(u', z'): z'\leq c^*\}$ for all $u\in\R, z\leq c^*$. As a result, $\{(u', z'): z'\leq c^*\}$ is a invariant control set for \eqref{e3.3}.
The uniqueness of this invariant control set is obtained.
in the property that $\{(u', z'): z'\leq c^*\}\subset\bar{{\cal O}^+_2(u, z)}$ for every $(u, z)\in\R^2$.
Equivalently, $C:=	\{(u, v): v-\dfrac\beta\alpha u\leq c^*\}$ is a unique invariant control set for \eqref{e3.2}.
\end{proof}

Note that if $\lambda>0$, there is an invariant probability measure $\pi^*$ of \eqref{e3.1} that is associated with $\mu^*$ of \eqref{e3.0}.
Since there is only one invariant control set $C$, it follows from Assumption \ref{asp3.1} that $\pi^*$ is the unique invariant probability measure with support $C$. Moreover, for all $(u,v)\in C$ and a $\pi^*$-integrable function $f$ we have
\begin{equation}\label{slln}
\PP\Big\{\lim\limits_{t\to\infty}\dfrac1t\int_0^tf\big(\xi^{u,v}(s), \eta^{u,v}(s)\big)ds=\int_{\R^2}f(u',v')\pi^*(du', dv')\Big\}=1.
\end{equation}
These results are proved in \cite{WK}.
Moreover, it follows from \cite[Proposition 5.1]{IK}
\begin{equation}\label{etv}
\lim\limits_{t\to\infty}\|\hat P(t, (u,v), \cdot)-\pi^*(\cdot)\|\to0\,\forall(u,v)\in C,
\end{equation} where $\|\cdot\|$ is the total variation norm, if we can verify the following H\"ormander condition.

\begin{asp}\label{asp3.2}
The ideal $\L_0$ in $\L$ generated by $B$ satisfies dim$\L_0(u, v)=2$ at every $(u, v)\in C$. In other words, the set of vectors $B, [A, B], [B, [A, B]], [B, [B, A, B]],\dots$ spans $\R^2$.
\end{asp}

We  aim to prove that \eqref{slln} (under Assumption \ref{asp3.1}) and \eqref{etv} (under Assumption \ref{asp3.2}) hold for all $(u,v)\in\R^2$.
We
need only consider the case $\beta<0$ or $\beta\geq\alpha$ since  $C=\R^2$ in case $0<\beta<\alpha$.

\begin{prop}\label{prop3.2}
Suppose that ${\beta}\geq\alpha, \lambda>0$. Then, for each initial value $(u,v)\in\R^2$, we have $\tau^{u,v}_{C^\circ}$ almost surely with $\tau^{u,v}_{C^\circ}=\inf\{t>0: (\xi^{u,v}(t),\eta^{u,v}(t))\in {C^\circ}\}$.
\end{prop}

The proof of this proposition is divided into several lemmas.
We consider only the case $c^*<\infty$ since the assertion is trivial if $c^*=\infty$.
Let us first explain the idea of the proof. Denote $d_1=\ln H, d_2=\ln\hbar$, where $\hbar, H$ are defined as in the proof of Theorem \ref{thm2.2}.
Since the process is recurrent relative to $\hat A:=\{(u,v): u\leq d_1, d_2\leq v\leq d_1\}$, in order to show $\tau^{u,v}_{C^\circ}<\infty$, we need to estimate
(uniformly)
the probability of entering $C^\circ$ from $\hat A$.
The difficulty is that $\hat A$ is not compact. Therefore, we divide $\hat A$ into $\hat A_1=\{(u,v): u< d_5, d_2\leq v\leq d_1\}$ and $\hat A_2=\hat A\setminus\hat  A_1$, where $-d_5$ is sufficiently large.
Noting that
$\hat A_2$ is compact and using the support theorem and the Feller property, we can
obtain a positive lower bound for the probability of entering $C$ from $\hat A_2$.
To obtain similar result for $\hat A_1$, we will analyze the property of the drift when $-u$ is sufficiently large and then estimate using the exponential martingale inequality.

Fix $0<\delta<\min\{a_1-{\alpha^2}/2, a_2+{\beta^2}/2\}$. Thus, there is  a $d_3<d_2$ such that
for all $u\leq\alpha\beta^{-1}(d_3-c^*), v\leq d_3$, we have
\begin{equation}\label{e3.3a}
a_1-\dfrac{\alpha^2}2-b_1e^{u}-\dfrac{c_1e^{v}}{m_1+
m_2e^{u}+m_3e^{v}}\geq\delta\mbox{ and }
-a_2-\dfrac{\beta^2}2-b_2e^{v}+\dfrac{c_2e^{u}}{m_1+m_2e^{u}+m_3e^{v}}\leq-\delta.
\end{equation}
Let $d_4\leq\min\{\dfrac\alpha\beta(d_3-c^*), d_3\}-\ell$ where $\ell>0$ be chosen such that $2\exp(-\dfrac{\delta\ell}{(\alpha+\beta)^2})<1$.
Construct open sets $D=\{(u,v)\in\R^2: u< \dfrac\alpha\beta(d_3-c^*), v< d_3\}$ and $E=\{(u,v)\in\R^2, u,v\leq d_4\}.$ Then
put $E_1=E^\circ\cap C^\circ$, $E_2=E\setminus E_1$.

\begin{lm}\label{lm3.4}
Suppose that $\beta\geq\alpha$. There is a $\tilde p>0$ such that
$$\PP\{\xi^{u,v}(\sigma^{u,v}_{D})=\dfrac\alpha\beta(d_3-c^*), \eta^{u,v}(\sigma^{u,v}_{D})<d_3\}\geq\tilde p_1\,\forall\, (u,v)\in E,$$
where $\sigma^{u,v}_{D}$ is the first time $(\xi^{u,v}(t),\eta^{u,v}(t))$ exits $D$.
\end{lm}

\begin{proof}
Define $\hat T_{u,v}=\dfrac2\delta\big(\dfrac\alpha\beta(d_3-c^*)-u+\ell\big)$.
By the well-known exponential martingale inequality, we have $\PP(\Omega_1)>\tilde p_1:=1-2\exp(-\dfrac{\delta\ell}{(\alpha+\beta)^2})$,
where $$\Omega_1:=\Big\{\omega:\sup_{0\leq t\leq\hat T_{u,v}}\big\{|W(t)|-\dfrac{\delta}{2(\alpha+\beta)}t\big\}<\dfrac\ell{\alpha+\beta}\Big\}.$$
For $\omega\in\Omega_1$ and $u, v\leq d_4$, it follows from the property of $\Omega_1$ and \eqref{e3.1} that
\begin{equation}\label{e3.3b}
\begin{aligned}
\xi^{u,v}(\sigma_D^{u,v}\wedge\hat T_{u,v})\geq& u+\delta(\sigma_D^{u,v}\wedge\hat T_{u,v})-\dfrac{\alpha\delta}{2(\alpha+\beta)}(\sigma_D^{u,v}\wedge\hat T_{u,v})-\dfrac{\alpha\ell}{\alpha+\beta}\\
\geq&u-\ell+\dfrac{\delta}2(\sigma_D^{u,v}\wedge\hat T_{u,v}),
\end{aligned}
\end{equation}
and that
\begin{equation}\label{e3.3c}
\begin{aligned}
\eta^{u,v}(\sigma_D^{u,v}\wedge\hat T_{u,v})\leq d_4-\delta(\sigma_D^{u,v}\wedge\hat T_{u,v})+\dfrac{\beta\delta}{2(\alpha+\beta)}(\sigma_D^{u,v}\wedge\hat T_{u,v})+\ell< d_3.
\end{aligned}
\end{equation}
If $\sigma_D^{u,v}> T_{u,v}$, it follows from \eqref{e3.3b} that
$\xi^{u,v}(\hat T_{u,v})\geq u-\ell +\dfrac{\delta}2\hat T_{u,v}\geq \dfrac{\alpha}\beta(d_3-c^*)$ which is a contradiction.
Hence $\sigma_D^{u,v}\leq T_{u,v}$ for all $\omega\in\Omega_1$.
Furthermore, \eqref{e3.3c} implies that for $\omega\in\Omega_1$, $\eta^{u,v}(\sigma_D^{u,v})=\eta^{u,v}(\sigma_D^{u,v}\wedge\hat T_{u,v})<d_3$ and consequently $\xi^{u,v}(\sigma_D^{u,v})=\dfrac{\alpha}\beta(d_3-c^*)$.
As a result,
$$\PP\big\{\xi^{u,v}(\sigma^{u,v}_{D})=\dfrac\alpha\beta(d_3-c^*), \eta^{u,v}(\sigma_D^{u,v})<d_3\big\}\ \geq\PP(\Omega_1)\geq\tilde p_1\,\forall (u,v)\in E.$$
The lemma is proved. \end{proof}

\begin{lm}\label{lm3.2}
Suppose that $\beta\geq\alpha$. There are $d_5\in\R$, $\tilde p_2>0$ and $\bar T>0$ such that
$$\PP\{\tau^{u,v}_{E}\leq \bar T\}\geq\tilde p_2\,\forall u\leq d_5, d_2\leq v\leq d_1,$$ where
$\tau^{u,v}_E$ is the first time $(\xi^{u,v}(t), \eta^{u,v}(t))$  enters $E$.

\end{lm}

\begin{proof}
It is readily seen that there are  $\sigma_1<d_4$, $G_1>0$ and $\delta_1>0$ such that
$$\sup\limits_{u\leq\sigma_1,v\in\R}\{a_1-\dfrac{\alpha^2}2-b_1e^{u}-\dfrac{c_1e^{v}}{m_1+
m_2e^{u}+m_3e^{v}}\}\leq G_1,$$
and that
$$\sup\limits_{u\leq\sigma_1,v\in\R}\{-a_2-\dfrac{\beta^2}2-b_2e^{v}+\dfrac{c_2e^{u}}{m_1+m_2e^{u}+m_3e^{v}}\}<-\delta_1.$$
Fix $\delta_2>0$. Define $\bar T=2\dfrac{d_1-d_4+\delta_2}{\delta_1}$ and $d_5=\sigma_1-\delta_2-(G_1+\dfrac{\delta_1}2)\bar T$ and the stopping time
$$\zeta^{u,v}=\inf\{t>0: \xi^{u,v}(t)\geq\sigma_1\mbox{ or } \eta^{u,v}(t)\leq d_4\}.$$
By the exponential martingale inequality, we have $\PP\{\Omega_2\}>\tilde p_2:=1-\exp(-\dfrac{\delta_1\delta_2}{(\alpha+\beta)^2})>0$
where $$\Omega_2:=\Big\{\omega:\sup_{0\leq t\leq\bar T}\big\{W(t)-\dfrac{\delta_1}{2(\alpha+\beta)}t\big\}<\dfrac{\delta_2}{\alpha+\beta}\Big\}.$$
For $\omega\in\Omega_2$ and $u<d_5, d_2\leq v\leq d_1$, it follows from the property of $\Omega_2$ and \eqref{e3.1} that
\begin{equation}\label{e3.4}
\begin{aligned}
\xi^{u,v}(\zeta^{u,v}\wedge\bar T)<& u+G_1(\zeta^{u,v}\wedge\bar T)+\dfrac{\alpha\delta_1}{2(\alpha+\beta)}(\zeta^{u,v}\wedge\bar T)+\dfrac{\alpha\delta_2}{\alpha+\beta}\\
\leq&d_5+\delta_2+(G_1+\dfrac{\delta_1}2)\bar T=\sigma_1,
\end{aligned}
\end{equation}
and that
\begin{equation}\label{e3.5}
\begin{aligned}
\eta^{u,v}(\zeta^{u,v}\wedge\bar T)<& d_1-\delta_1(\zeta^{u,v}\wedge\bar T)+\dfrac{\beta\delta_1}{2(\alpha+\beta)}(\zeta^{u,v}\wedge\bar T)+\dfrac{\beta\delta_2}{\alpha+\beta}\\
\leq&d_1+\delta_2-\dfrac{\delta_1}2(\zeta^{u,v}\wedge\bar T).
\end{aligned}
\end{equation}

If $\zeta^{u,v}>\bar T$, we deduce from \eqref{e3.5} that
$\eta^{u,v}(\bar T)< d_1+\delta_2-\dfrac{\delta_1}2(\bar T)=d_4,$ which contradicts the definition of $\zeta^{u,v}$.
Hence for $\omega\in\Omega_2$, we have $\zeta^{u,v}\leq\bar T$.
Moreover, \eqref{e3.4} implies that $\xi^{u,v}(\zeta^{u,v})<\sigma_1$. In view of the definition of $\zeta^{u,v}$, we have
$\eta^{u,v}(\zeta^{u,v})=d_4$ in $\Omega_2$, consequently $\tau^{u,v}_E\leq\bar T$ in $\Omega_2$.
As a result, for any $u\leq d_5, d_2\leq v\leq d_1$, $\PP\{\tau_E^{u,v}\leq\bar T\}\geq\PP(\Omega_2)\geq \tilde p_2.
$
\end{proof}

\begin{lm}\label{lm3.3}
Suppose that ${\beta}\geq\alpha, \lambda>0$. For any $(u,v)\in\R^2$, the process $(\xi^{u,v}(t),\eta^{u,v}(t))$ is recurrent relative to $E$,
that is,  there is a sequence of random variables $\{t_n(\omega)\}$ such that $t_n(\omega)\uparrow\infty$ as $n\to\infty$ and that $(\xi^{u,v}(t_n),\eta^{u,v}(t_n))\in E\,\forall n\in\N$ for almost all $\omega$.
\end{lm}

\begin{proof}
Since $E_1\subset \bar{{\cal O}^+_1(u, v)}\,\forall\, (u, v)\in\R^2$, it follows from the support theorem (see \cite[Theorem 8.1, page 518]{IW} or \cite{SV})
for diffusion processes, that there is a $T_{u, v}>0$ such that $\PP\big\{\big(\xi^{u,v}(T_{u,v}), \eta^{u,v}(T_{u,v})\big)\in E_1\big\}>2p^{u,v}>0$.
Since the process $(\xi(t),\eta(t))$ is Feller and $E_1$ is an open set, there is a neighborhood $V_{u,v}$ of $(u,v)$ such that for
$\PP\big\{\big(\xi^{u',v'}(T_{u,v}), \eta^{u',v'}(T_{u,v})\big)\in E_1\big\}>p_{u,v}\,\forall (u',v')\in V_{u,v}.$
Let $d_5$ be as in Lemma \ref{lm3.2},  we consider the compact set $K=\{(u,v): d_5\leq u\leq d_1, d_2\leq v\leq d_1\}$. By the Heine-Borel theorem, there is a finite number of $V_{u_i, v_i}, i=1,
\dots,n$ such that $K\subset\cup_{i=1}^nV_{u_i,v_i}$. Letting $\bar T_K=\max\{T_{u_i,v_i}, i=1,n\},$ $\bar p_K=\min\{p_{u_i,v_i}, i=1,n\}$. we claim that for any $(u,v)\in K$, $\PP\{\tau^{u,v}_{E}\leq T_K\}\geq\PP\{\tau^{u,v}_{E_1}\leq T_K\}\geq p_K>0.$
Combining this result with the conclusion of Lemma \ref{lm3.2}, we derive that there are  $\hat T>0$, $\hat p>0$ such that
\begin{equation}\label{e3.6}
\PP(\tau^{u,v}_E<\hat T)\geq\hat p\,\forall\, (u,v)\in \hat A:=\{(u,v): u\leq d_1, d_2\leq v\leq d_1\}.
\end{equation}
Since \eqref{e2.18} is equivalent to
$$\dfrac1t\int_0^t\1_{\{(\xi^{u,v}(s),\eta^{u,v}(s))\in\hat A\}}ds>0 \ \hbox{ a.s.,} \,\forall\, (u,v)\in\R^2,$$
the process $(\xi^{u,v}(t),\eta^{u,v}(t))$ is recurrent relative to $\hat A$.
Using this property, the strong Markov property and \eqref{e3.6}, we can conclude the recurrence relative to $E$ of $(\xi^{u,v}(t),\eta^{u,v}(t))$.
\end{proof}

\begin{proof}[Proof of Proposition \ref{prop3.2}]
Since $(\xi^{u,v}(t),\eta^{u,v}(t))$ is recurrent relative to $\hat A$ and $E$, we can
define the
following sequences of stopping times
\begin{align*}
\varsigma_1=&\inf\{t>0: \xi^{u,v}(t),\eta^{u,v}(t))\in E\},\\
\upsilon_n=&\inf\{t>\varsigma_n: \xi^{u,v}(t),\eta^{u,v}(t))\in \hat A\},\\
\varsigma_{n+1}=&\inf\{t>\upsilon_n: \xi^{u,v}(t),\eta^{u,v}(t))\in E\},
\end{align*}
which
 are finite almost surely.

We also define $\iota_n=\inf\{t>\varsigma_n: \xi^{u,v}(t),\eta^{u,v}(t))\notin D\}$.
Since $E\subsetneq D\subsetneq \hat A^c$, it is easy to see that $\varsigma_n<\iota_n<\upsilon_n$.
Consider a sequence of events $O_n:=\{\xi^{u,v}(\iota_n)=\dfrac\alpha\beta(d_3-c^*), \eta^{u,v}(\iota_n)<d_3\}.$
If we are in the time $\varsigma_n$, $O_n$ is the future information while we have already known whether $O_{n-1}$ has happened.
Moreover, it follows from Lemma \ref{lm3.4} that $\PP\big(O_n^c|\xi^{u,v}(\varsigma_n)=u',\eta^{u,v}(\varsigma_n)=v'\big)\leq1-\tilde p_1\,\forall (u',v')\in E$.
Hence, using the strong Markovian property of $(\xi^{u,v}(t), \eta^{u,v}(t))$, we can prove that
$$\PP\Big(\bigcap_{k=1}^nO^c_k\Big)\leq(1-\tilde p_1)^n\to0\mbox{ as } n\to\infty.$$
This means that almost surely, $O_n$ must occur for some $n=n(\omega)$.
Whenever $O_n$ occurs, we have $(\xi^{u,v}(\iota_n),\eta^{u,v}(\iota_n))\in C^\circ$. The proof is complete.
\end{proof}

For the case $\beta<0$, we have a similar result.

\begin{prop}\label{prop3.3}
Suppose $\beta<0$, $\lambda>0$. Then, for each initial data $(u, v)\in\R^2$,  $\tau^{u,v}_{C^\circ}<\infty$ almost surely.
\end{prop}

\begin{proof}
We  only consider the case $c^*<\infty$ for which $C=\{(u,v): v\leq c^*-ru\}$ with $r=-\frac\beta\alpha>0$.
Let $\hat A$ be as in the proof of Lemma \ref{lm3.3}.
Divide $\hat A$ into $\hat A_1$ and $\hat A_2$ defined by
$\hat A_1=\hat A\cap C^\circ\,\mbox{ and } \hat A_2=\hat A\setminus \hat A_1.$
It is easy to see that $\hat A_2$ is compact.
Using the same arguments as in the proof of Lemma \ref{lm3.3}, we can find $\bar T_{\hat A_2}>0$ such that
$\inf_{(u',v')\in \hat A_2}\PP(\tau^{u',v'}_{C^\circ}<\bar T_{\hat A_2})>0.$
Since $\hat A_1\subset C^\circ$, we have
$$\inf_{(u',v')\in \hat A}\PP(\tau^{u',v'}_{C^\circ}<\bar T_{\hat A_2})=\inf_{(u',v')\in \hat A_2}\PP(\tau^{u',v'}_{C^\circ}<\bar T_{A_2})>0.$$
Moreover, since $(\xi^{u,v}(t),\eta^{u,v}(t))$ is recurrent relative to $\hat A$,
we can use the strong Markov property to obtain the desired conclusion.
\end{proof}

We complete this section by presenting the following theorem.

\begin{thm}\label{thm3.1}
Suppose $\alpha,\beta\ne0$, $\lambda>0$, and Assumption \ref{asp3.1} holds. Then, \eqref{e3.1} has a unique
invariant probability measure $\pi^*$ satisfying that for any $\pi^*$-integrable function $f$,
\begin{equation}\label{slln1}
\PP\Big\{\lim\limits_{t\to\infty}\dfrac1t\int_0^tf\big(\xi^{u,v}(s), \eta^{u,v}(s)\big)ds=\int_{\R^2}f(u',v')\pi^*(du', dv')\Big\}=1 \  \ \forall (u,v)\in\R^2.
\end{equation}
Moreover, if Assumption \ref{asp3.2} is satisfied, the transition probability $\hat P(t, (u,v),\cdot)$ converges to $\pi^*(\cdot)$ in total variation as $t\to\infty$.
\end{thm}

\begin{proof}
The  assertions can be proved using \eqref{slln}, \eqref{etv}, Propositions \ref{prop3.2}, and \ref{prop3.3}.
\end{proof}

\section{Discussion}\label{sec:dis}
We compare our results with some of the recent results in the literature.
In \cite[Theorem 4.1]{JJ}, under the conditions $\dfrac{c_2}{m_2}<a_2+\dfrac{\beta^2}2$ and $a_1> \alpha^2/2$, it was proved that the predator will eventually die out while the distribution of $x(t)$ converges weakly to the stationary distribution of $u(t)$.
In contrast, using Theorem \ref{thm2.1} of this paper, we obtain the same conclusion provided that
 $a_1> \alpha^2/2$ and $\lambda<0$.
Note that $\lambda<0$ is equivalent to $$\tilde\lambda:=\int_{0}^\infty\dfrac{c_2x}{m_1+m_2x}\mu_-(dx)<a_2+\beta^2/2.$$
It is easy to verify that $\tilde\lambda<\dfrac{c_2}{m_2},$ which
indicates that our result on extinction of predator is sharper.
Furthermore,
a suitable Lyapunov function was used in \cite{JJ} to obtain the ergodicity of system \eqref{me} for the non-degenerate case as follows (see \cite[Theorem 3.1]{JJ}).

\begin{thm}\label{e5.1}
Assume $(c_2-a_2m_2)a_1/b_1>a_2m_1,b_1>a_1m_2/(m_1+m_2x^*)$ and $\alpha>0, \beta>0$ such that $ \delta<\min\{c_2(b_1-m_2(a_1-b_1x^*)/m_1)(m_1+m_3y^*)(x^*)^2,b_2c_1(m_1+m_2x^*)(y^*)^2\}$, where $\delta=c_2x^*\alpha^2/2+c_1y^*\beta^2/2$ and $(x^*,y^*)$ is the equilibrium of the deterministic system
\begin{equation}\label{e}
\begin{cases}
\dot x(t)=x(t)\big(a_1-b_1x(t))-\dfrac{c_1y(t)}{m_1+m_2x(t)+m_3y(t)}\big)dt, \\
\dot y(t)=\big(-a_2-b_2y(t)+\dfrac{c_2x(t)}{m_1+m_2x(t)+m_3y(t)}\big)dt.
\end{cases}
\end{equation}
 Then there is a stationary distribution $\pi(\cdot)$ for system (1.2) and it has ergodic property.
\end{thm}

To show that their assumption is more restrictive than our assumption of ergodicity,
let $G$ be the space of the positive parameters $(a_i, b_i, c_i, m_j, \alpha,\beta)$, $i=1,2$, $j=1, 2, 3,$ $a_1>\alpha^2/2$,
 and $$G^+=\{(a_i, b_i, c_i, m_j, \alpha,\beta): \lambda>0\},\  \ G^-=\{(a_i, b_i, c_i, m_j, \alpha,\beta): \lambda<0\}.$$
It is easy to check that $\lambda$ is a continuous function of parameters. Hence $G^+$ and $G^-$ are open. Moreover, the closure
cl$(G^-)=\{\lambda\leq0\}=(G^+)^c$, which is a necessary condition for the extinction of the predator.
Let $J$ be the set of parameters satisfying the assumption of Theorem \ref{e5.1}, we must have $G^-\cup J=\emptyset$. Since $J$ is open, cl$(G^-)\cup J=\emptyset$ or equivalently $J\subset G^+$.

We will show that $J$ is a proper subset of $G^+$.
Choose $a_1, b_1, c_1, a_2, c_2, m_i, i=1,3, \alpha, \beta$ such that $\lambda>0$. This choice can be done by taking $a_1$ sufficiently large. Now fix these parameters. Since $\lambda$ does not depend on $b_2$, we claim the ergodicity holds for all $b_2>0$.
It can be proved that there exists $M>0$ independent of $b_2$ such that $x^*, y^*<M$, where  $(x^*, y^*)$ is the positive equilibrium of \eqref{e} (if it exists). Thus, for sufficiently small $b_2$ such that
$ \delta>b_2c_1(m_1+m_2x^*)(y^*)^2$, the assumption of Theorem \ref{e5.1} does not hold while $\lambda>0$.

Next we look at the case $m_1=1$, $m_2=1$, $m_3=0$ for which the functional response is said to be Holling type-II (see \eqref{e1.0}).
We will make a comparison with the findings in \cite{LW} in which they proved that if $a_1-\dfrac{\alpha^2}2>0$ and $c_2+a_2-\dfrac{\beta^2}2<0$, the predator will extinct while $x(t)$ converges weakly to the stationary distribution of $\phi(t)$. Moreover, it was shown that the system is persistent in time-average if
$$a_1-\dfrac{\alpha^2}2>0,\  a_2-\dfrac{\beta^2}2>0, \ \hbox{  and } \ \dfrac{a_1-\frac{\alpha^2}2}{c_1}>\dfrac{c_2+a_2-\frac{\beta^2}2}{b_2}.$$
In the same manner as
in the previous part, we can show that our conditions for extinction or permanence and ergodicity are
weaker than those in \cite{LW}.

We have investigated \eqref{me} and \eqref{e3.0} when $\lambda\ne0$.
Note that the set $\{\lambda=0\}$ has  Lebesgue measure zero in the space of parameters $G$.
Although the set $\{\lambda=0\}$ is negligible with respect to the Lebesgue measure,
it is still interesting to explore the asymptotic behavior of the solution in this critical case.
The question of asymptotic behavior corresponding to $\lambda=0$
remains open. To treat this case, new techniques are needed. Moreover,
 it seems that our methods are applicable to stochastic predator-prey models with different types of functional responses as well as different diffusion coefficients. Furthermore, our method can be applied to stochastic models with Markovian switching.

\

{\bf Acknowledgment.}  We  gratefully thank
the reviewer for
constructive comments
and detailed suggestions, which led to
much improvement
of the paper.

\end{document}